\documentclass[a4paper,reqno,index]{amsart}

\usepackage[normalem]{ulem} 

\DeclareMathAlphabet{\lcal}{U}{dutchcal}{m}{n}
\usepackage{amsmath, amsfonts, amsthm, amssymb}
\usepackage[maxbibnames=9]{biblatex}
\usepackage{mathrsfs}
\usepackage{tikz}
\usetikzlibrary{arrows.meta, positioning, calc}
\usepackage[title]{appendix}
\usepackage{amsthm}
\usepackage{xcolor}
\usepackage{graphicx}
\usepackage{hyperref}
\usepackage{mathtools}

\usepackage{booktabs}
\usepackage{array}
\usepackage{longtable}

\usepackage[shortlabels]{enumitem}

\usepackage{setspace}

\newcommand{\df}[1]{{\textit{#1}}{\index{#1}}}
 
 \numberwithin{equation}{section}

\allowdisplaybreaks

\makeindex

\hypersetup{
    bookmarks=true,         
    unicode=false,          
    pdftoolbar=true,        
    pdfmenubar=true,        
    pdffitwindow=false,     
    pdfstartview={FitH},    
    pdftitle={My title},    
    pdfauthor={Author},     
    pdfsubject={Subject},   
    pdfcreator={Creator},   
    pdfproducer={Producer}, 
    pdfkeywords={keyword1, key2, key3}, 
    pdfnewwindow=true,      
    colorlinks=false,       
    linkcolor=green,          
    citecolor=green,        
    filecolor=green,      
    urlcolor=green,           
    urlbordercolor={1 1 1}  
}

\newtheorem{theorem}{Theorem}[section]
\newtheorem*{theorem*}{Theorem}
\newtheorem*{conjecture*}{Conjecture}
\newtheorem{lemma}[theorem]{Lemma}

 \newcommand{\Z}{\mathbb{Z}}

\newcommand{\HM }{\mathcal{M}}
\newcommand{\HN }{\mathcal{N}}

\newcommand{\Mult}{\operatorname{Mult}}

\newcommand{\cB}{\mathcal{B}}
\newcommand{\cM}{\mathcal{M}}

\newcommand{\HE}{\mathcal{E}}
\newcommand{\HF}{\mathcal{F}}

\newcommand{\cH}{\mathcal{H}}
\newcommand{\B}{\mathbb{B}}

\newcommand{\HSF}{\mathcal{F}}
\newcommand{\HSE}{\mathcal{E}}

\newcommand{\ol}[1]{\overline{#1}}

\newtheorem{proposition}[theorem]{Proposition}
\newtheorem*{claim}{Claim}

\newtheorem{corollary}[theorem]{Corollary}
\theoremstyle{definition}
\newtheorem{definition}[theorem]{Definition}

\newtheorem{remark}[theorem]{Remark}

  \usepackage{etoolbox}

\begin{document}

\newcommand{\C}{\mathbb{C}}

\newcommand{\N}{\mathbb{N}}
\newcommand{\T}{\mathbb{T}}

\theoremstyle{definition}

\theoremstyle{remark}

\title[\small WHEN DO KERNELS ADMIT  CHARACTERISTIC FUNCTIONS?]{\large   WHEN DO KERNELS ADMIT  CHARACTERISTIC FUNCTIONS?}

\author[\small Gu]{Caixing Gu}
\address{Department of Mathematics, California Polytechnic State University, San Luis Obispo, CA 93407, USA}
\email{cgu@calpoly.edu}

\author[\small Luo]{Shuaibing Luo}
\address{School of Mathematics, Hunan University, Changsha, Hunan 410082, China}
\email{sluo@hnu.edu.cn} 

\author[\small Tsikalas]{Georgios Tsikalas}
\address{Department of Mathematics, UC Santa Barbara, Santa Barbara, CA}
\email{gtsikalas@ucsb.edu}

\author[\small Zhou]{Min Zhou}
\address{School of Mathematics, Hunan University, Changsha, 410082, PR China}
\email{minzhou1215@hnu.edu.cn}

 
\begin{abstract}
A general framework for deriving characteristic functions for reproducing kernels that do not necessarily possess the complete Pick property was recently established by Bhattacharyya and Jindal. 
We show that, in this setting, the existence of a characteristic function is equivalent to a Beurling-type invariant subspace condition. Combined with recent results characterizing kernels satisfying this condition, our theorem implies that the existence of a characteristic function is equivalent to a concrete Agler-type decomposition of the underlying kernels.
  \end{abstract}
\maketitle

 \par

\normalsize

\section{Introduction} \label{introsec}

\onehalfspacing

Let $\B_d = \{z = (z_1, \ldots, z_d) \in \C^d: |z| < 1\}$ be the unit ball in $\C^d$, with $\Z_+$   the set of all non-negative integers. Let $K_w(z)$ be a \textit{diagonal holomorphic} kernel on $\B_d$, meaning
\begin{align*}
K_w(z) = K(z,w) = \sum_{\alpha\in \Z^d_+} k_{\alpha} z^\alpha \overline{w}^\alpha, \qquad z, w \in \B_d,
\end{align*}
where $k_{0} = 1, k_{\alpha}  \geq 0, \alpha = (\alpha_1, \ldots, \alpha_d)\in \Z^d_+$  and $ z^\alpha = z_1^{\alpha_1}\cdots z_d^{\alpha_d}$. 
Since $k_0 = 1$, there exist real numbers $\{b_{\alpha}\}$ and $\varepsilon > 0$ such that
\begin{align}\label{kernelcoeffbn}
1 - \frac{1}{K_w(z)} = \sum_{\alpha\in \Z^d_+\backslash \{0\}} b_\alpha  z^\alpha \overline{w}^\alpha,\qquad  |z|, |w|< \varepsilon.
\end{align}
When $K_w(z) \neq 0$ on $\B_d \times \B_d$, the above equality holds for all $z, w \in \B_d$. Let the corresponding reproducing kernel Hilbert space be denoted by $\mathcal{H}_K$.

Given a second diagonal holomorphic kernel $L$ over  $\mathbb{B}_d$ and  coefficient Hilbert spaces $\HSE, \HSF$,
a function $\Phi:\mathbb{B}_d\to \cB(\HSF,\HSE)$ is a \df{multiplier} from $\cH_K \otimes \HSF$ to $\cH_L\otimes\HSE$
 if $\Phi F\in \cH_L \otimes \HSE$ whenever $F\in \cH_K\otimes \HSF.$ $\Phi$ induces a bounded operator $M_\Phi:\cH_K\otimes \HSF\to \cH_L\otimes \HSE$ via 
$M_\Phi F= \Phi \cdot F$.  The set  \df{$\Mult(\cH_K\otimes \HSF,\cH_L\otimes \HSE)$} of such multipliers
is naturally a subspace of $\cB(\cH_K\otimes \HSF,\cH_L\otimes \HSE)$. A multiplier $\Phi$ is called
\df{inner}    if $M_{\Phi}$ is partially isometric. Let $M_z=(M_{z_1}, \dots, M_{z_d})$ denote the $d$-tuple of multiplication by the coordinate functions. We will always assume that $M_z$ is bounded on $\mathcal{H}_K$ and $\mathcal{H}_L,$ and we will say that a closed subspace $\cM\subseteq \mathcal{H}_{L}\otimes\mathcal{E}$ is \df{shift-invariant} if $(M_z\otimes I_{\mathcal{E}})\mathcal{M}\subset\mathcal{M}$.

 Now, let $H$ be a separable Hilbert space. Given a bounded commuting $d$-tuple $T = (T_1, T_2, \ldots, T_d)$  on $H$, write $T^{\alpha}=T^{\alpha_1}\cdots T^{\alpha_d}$ for any multi-index $\alpha.$
 \begin{definition}\label{1/k-def}
Given a diagonal holomorphic kernel $K$ on $\mathbb{B}_d$ and a bounded commuting $d$-tuple $T,$ we say that $T$ is a 
\begin{itemize} 
    \item[(i)] \textit{$1/K$-contraction}  if $\sum_{\alpha\in \Z^d_+\backslash \{0\}} b_\alpha T^\alpha T^{*\alpha}$ converges in the strong operator topology and
$I - \sum_{\alpha\in \Z^d_+\backslash \{0\}} b_\alpha  T^\alpha T^{*\alpha} \geq 0;$
\item[(ii)] \textit{pure $1/K$-contraction} if $T$ is a $1/K$-contraction and $\sum_{\alpha\in \Z^d_+} k_\alpha T^\alpha \Delta_T^2 T^{*\alpha}$
converges in the strong operator topology to $I$, where $\Delta_T$ is the positive square root of $I - \sum_{\alpha\in \Z^d_+\backslash \{0\}} b_\alpha  T^\alpha T^{*\alpha}$.
\end{itemize}
 \end{definition}
\noindent The genesis of the concept of a $1/K$-contraction can be traced back to the work of Agler \cite{coanalytic, Aglerhypercontractions}. Further developments can be found in \cite{Pottmodels, EschmeierToth, Yakmodels, Clouatre-Timko, Polynomialarazyenglis, ArazyTAMS, jindalkumar}. Note that, if $K(z,w)=(1-\langle z, w\rangle)^{-1}$ is the Drury-Arveson kernel, then $T$ is a $1/K$-contraction if and only if it is a row contraction. In particular, if $\mathcal{H}_K$ is the Hardy space over the unit disk, then $1/K$-contractions and contractions coincide.

In the classical setting, the Sz.-Nagy–Foiaș characteristic function attached to a contraction is realized as an inner multiplier on a vector-valued Hardy space. Together with its associated operator model, it has revealed deep connections among shift-invariant subspaces, operator-valued inner functions, conservative linear systems, and the structure theory of contractions \cite{Nikolski-spectraltheory, Foias-Bercovici-harmonic}. While this classical theory is tied to the Hardy space and its multipliers, in \cite{Charadvances} the notion of a characteristic function was extended to  \textit{complete Pick spaces}. These are spaces that host an analogue of the classical Pick
interpolation theorem for multipliers. They have been extensively studied and play a fundamental role in multivariable operator theory; notable examples beyond the Szegő kernel include the Dirichlet and Drury-Arveson kernels. A standard reference is \cite{Pickbook}.

 It turns out that the existence of a characteristic function as in \cite{Charadvances} is equivalent to the complete Pick property itself. Thus, outside the complete Pick setting, a characteristic function can no longer live on a single space; it must be realized as a multiplier between two spaces, which makes pairs of kernels its natural habitat.  Remarkable efforts to define a characteristic function in this direction can be found in  \cite{OlofssonChar, EschmeirBergman, EschmeierToth, BhatarHypercontr, Bhattardilations}. In \cite{CharacteristicPickfactor}, a general framework was developed that unified these works. Before we present it, we record the following relevant construction. Given a pure $1/K$-contraction $T$ on $H$,  define $V_T: H \rightarrow \mathcal{H}_K \otimes \overline{\operatorname{ran}} \Delta_T$ by
\begin{align}\label{VT}
V_T h = \sum_{\alpha\in \Z^d_+} k_{\alpha} z^\alpha \otimes \Delta_T T^{*\alpha} h.
\end{align}
Then, for $p\in\mathbb{C}[z_1, \dots, z_d]$  and $x \in \overline{\operatorname{ran}} \Delta_T$, $V_T^*(p \otimes x) = p(T) \Delta_T x$. In particular, $V_T$ is an isometry and
$$V_T T_i^* = (M_{z_i} \otimes I_{\overline{\operatorname{ran}} \Delta_T})^* V_T, \qquad i = 1, 2, \ldots, d,$$
see \cite{ArazyTAMS, Polynomialarazyenglis} for details.  Note that $\ker V^*_T$ is invariant under $M_{z_i} \otimes I_{\overline{\operatorname{ran}} \Delta_T}$. For the definition of an admissible pair, see Section \ref{backgroundsec}.
 \begin{definition}\label{charactdef}
Let $(K, L)$ be an admissible pair of kernels on $\mathbb{B}_d.$
We say that \textit{$L$ admits a characteristic function through $K$} if, for every pure $1/L$-contraction $T$, there exist a Hilbert space $\mathcal{F}$ and a multiplier $\Phi \in \text{Mult}(\mathcal{H}_K\otimes \mathcal{F}, \mathcal{H}_L\otimes \overline{\operatorname{ran}} \Delta_T)$ such that
\[I - V_T V_T^* = M_\Phi M_\Phi^*.\]
 \end{definition}

    \begin{remark}
    In \cite{CharacteristicPickfactor}, the authors worked with unitarily invariant kernels. We adopt their definition in a somewhat more general setting. 
    \end{remark}
 By definition, the existence of a characteristic function requires that certain shift-invariant subspaces (those arising as kernels of the maps $V^*_T$) be representable as ranges of partially isometric multipliers $\Phi$ from $\mathcal{H}_K\otimes F$. Our main result says that this apparently partial requirement is in fact all-or-nothing: if these particular subspaces are so representable, then so is every shift-invariant subspace.

\begin{theorem} \label{mainthm}
Let $(K, L)$ be an admissible pair of kernels on $\mathbb{B}_d.$  
Then, $L$ admits a characteristic function through $K$ if and only if, for every separable Hilbert space $\mathcal{E}$ and closed shift-invariant $\mathcal{M}\subset\mathcal{H}_{L}\otimes \mathcal{E},$ there exists an auxiliary Hilbert space $\mathcal{F}$ and an inner multiplier $\Phi:\mathcal{H}_K\otimes\mathcal{F}\to \mathcal{H}_L\otimes\mathcal{E} $ such that 
\[\mathcal{M}=\Phi(\mathcal{H}_K\otimes\mathcal{F}).\]
\end{theorem}
In other words, following the terminology of \cite{BLHpairs}, $L$ admits a characteristic function through $K$ if and only if $(K, L)$ is a Beurling-Lax-Halmos pair. Via  the main result of \cite{BLHpairs}, and  under an additional regularity assumption, we thus obtain the following concrete description. For the definition of a strongly admissible pair of kernels, see Section \ref{backgroundsec}.

\begin{theorem} \label{BLHcorol}
Let $(K, L)$ be a strongly admissible pair of  kernels on $\mathbb{B}_d.$ Then,  $L$ admits a characteristic function through $K$ if and only if there exist Hilbert spaces $\mathcal{F}, \mathcal{G},$ a holomorphic function $Q:\mathbb{B}_d\to\mathcal{B}(\mathcal{G},\mathbb{C})$ and a contractive multiplier $\Phi:\mathcal{H}_L\otimes\mathcal{F}\to \mathcal{H}_L\otimes\mathcal{G}$ such that 
\begin{equation} \label{Aglerdecomp}
    \cfrac{1}{K(z,w)}=Q(z)\big[I-\Phi(z)\Phi(w)^* \big]Q(w)^*, \qquad z, w\in\mathbb{B}_d.
\end{equation}
\end{theorem}

Section \ref{backgroundsec} introduces the needed background and
terminology to make precise the statements of Theorems \ref{mainthm} and \ref{BLHcorol}. Section \ref{mainsec} contains the proofs of these results, as well as an example; see Proposition~\ref{exampleProp}.

\section{Preliminaries} \label{backgroundsec}

    Given a self-adjoint kernel $L$, we will write $L\succeq 0$ when  $L$ is positive semi-definite. The usual definition of a complete Pick kernel is in terms of certain multiplier interpolation problems. However, it can be shown that, at least for non-vanishing reproducing kernels $L$ on $\mathbb{B}_d$ satisfying $L(z, 0)=1$ for all $z,$ $L$ is complete Pick if and only if \[1-1/L\succeq 0.\]
   Some recent developments concerning complete Pick kernels can be found in
 \cite{AHMRfact, Chalmsimply, weakPickprod, ColRow, Charadvances, Wickcorona, DBrtoPick, mcctsik}. 
    
    We will work with reproducing kernels satisfying the following regularity assumptions. Note that we will write $M^K_z$ (resp. $M^L_z$) in place of $M_z$ when dealing with pairs of kernels, to indicate that the domain of $M_z$ is $\mathcal{H}_K$ (resp. $\mathcal{H}_{L}$).
\begin{definition}\label{admissible-def}
Given diagonal holomorphic kernels $K, L$ on $\mathbb{B}_d$, we say that $(K, L)$ is \textit{admissible} if
\begin{enumerate}[(i)]
    \item \label{K=nonzero} $K(z,w)\neq 0$ for $ z, w\in\mathbb{B}_d$;

    \item  $M^K_z$ is bounded on $\mathcal{H}_K$;

    \item   \label{contractionitem} $M^L_z$ is bounded on $\mathcal{H}_L$ and a $1/L$-contraction;

    \item  \label{Schurcompl}  For any $z\in\mathbb{B}_d$ and sequence $\{w_n\}\subset\mathbb{B}_d$ with $w_n\to\partial\mathbb{B}_d,$ we have 
    \[\cfrac{L(z, w_n)}{||L_{w_n}||}\to 0. \]
\end{enumerate}
\end{definition}

Regarding item \ref{contractionitem}, assume
$L$ is a kernel such that $S:=M^L_z$ is bounded, the Taylor spectrum of $S$ lies in $\overline{\mathbb{B}_d}$ and $1/L(z,\overline{w})$ extends analytically to a neighborhood of $\overline{\mathbb{B}_d}\times \overline{\mathbb{B}_d}$. Then, one can define $\sum_{\alpha } d_\alpha S^\alpha S^{*\alpha}$, where $1/L(z,w)=\sum_{\alpha}d_{\alpha}z^{\alpha}\overline{w}^{\alpha}$, via Taylor's functional calculus. It can, in fact, be verified that $S$ is a $1/L$-contraction in this case (see \cite[Section 2]{BLHpairs} and the computation in the proof of Lemma 6.4 of the same paper), and thus all such kernels satisfy item~\ref{contractionitem}.

We also require the following stronger admissibility condition, which connects to the ``strong regularity" of \cite{BLHpairs}. Let $\sigma(T)$ denote the Taylor spectrum of $T.$
\begin{definition} \label{strong-admissible-def}

Given a pair $(K, L)$ of diagonal holomorphic kernels  on $\mathbb{B}_d$, we say that $(K, L)$ is \textit{strongly admissible} if it is admissible and

\begin{enumerate}[(i)]
\item \label{S1} $\sigma(M^L_z)=\overline{\mathbb{B}_d}$;
\item    \label{S2} $\textup{Mult}(\mathcal{H}_{K})\subset \textup{Mult}(\mathcal{H}_{L})$;
\item  \label{S3}$L\neq 0$ on $\mathbb{B}_d\times\mathbb{B}_d$;
\item \label{S4} There exists $t>1$ and a sequence $0<r_n\nearrow 1$ such that, setting $K_n(z,w)=K(r_n z, r_n w),\ L_n(z,w)=L(r_n z, r_n w)$,
\begin{enumerate}[(a)]
    \item $1/K(z,\overline{w})$ extends analytically to $t\mathbb{B}_d\times t\mathbb{B}_d$;

\item $1/K_n\to 1/K$ uniformly on compact subsets of $t\mathbb{B}_d\times t\mathbb{B}_d$;

\item  \label{S4c} $L/L_n\succeq 0$, for all $n.$
\end{enumerate}
\end{enumerate}
\end{definition}

Examples of kernels satisfying condition~\ref{S4c} of item~\ref{S4}  include all kernels of the form $e^H,$ where $H$ is a diagonal holomorphic kernel. This class includes, in particular, products of unitarily invariant complete Pick kernels.

Finally, we record the following general lemma. For a proof, see e.g. \cite[Lemma 2.1]{Beurlingfactor}.

\begin{lemma} \label{coiso-mult}
Let $\mathcal{E}$ and $\mathcal{F}$ be Hilbert spaces, let $K$ be a
$B(\mathcal{E})$-valued kernel on a set $X$, let $L$ be a
$B(\mathcal{F})$-valued kernel on $X$, and let
$\Phi:X\to B(\mathcal{E},\mathcal{F})$ be a function.  
Then, $M_{\Phi}$ is a co-isometry from $\mathcal{H}_K$ to $\mathcal{H}_L$ if and only if
\[L(z, w) -  \Phi(z) K(z, w) \Phi(w)^* = 0, \qquad z ,w \in X.\]
\end{lemma}

\section{Main Results} \label{mainsec}

Given a closed subspace $\mathcal{M}\subset\mathcal{H}_L\otimes\HE,$ we let $L_{\HM}$ denote its reproducing kernel.

\begin{proof}[Proof of Theorem \ref{mainthm}]
First, assume that, for every separable Hilbert space $\mathcal{E}$ and closed shift-invariant $\mathcal{M}\subset\mathcal{H}_{L}\otimes \mathcal{E},$ there exists an auxiliary Hilbert space $\mathcal{F}$ and an inner multiplier $\Phi:\mathcal{H}_K\otimes\mathcal{F}\to \mathcal{H}_L\otimes\mathcal{E} $ such that 
$\mathcal{M}=\Phi(\mathcal{H}_K\otimes\mathcal{F}).$
Write $L(z,w)=\sum\ell_{\alpha}z^{\alpha}\overline{w}^{\alpha}$ and let $T$ be a pure $1/L$-contraction on $H$. As in \eqref{VT}, define the isometry $V_T: H \rightarrow \mathcal{H}_L \otimes \overline{\operatorname{ran}} \Delta_T$ by
\begin{align}\label{VTagain}
V_T h = \sum_{\alpha\in \Z^d_+} l_{\alpha} z^\alpha \otimes \Delta_T T^{*\alpha} h.
\end{align}
Since $V_T T_i^* = (M^L_{z_i} \otimes I_{\overline{\operatorname{ran}} \Delta_T})^* V_T, \ i = 1, 2, \ldots, d,$ it is clear that $\ker V^*_T $ is shift-invariant. By assumption, there exists an auxiliary Hilbert space $\mathcal{F}$ and an inner multiplier $\Psi:\mathcal{H}_K\otimes\mathcal{F}\to \mathcal{H}_L\otimes  \overline{\operatorname{ran}} \Delta_T $ such that
$\ker V^*_T=\Psi(\mathcal{H}_K\otimes\mathcal{F})$. Since $M_{\Psi}$ is a partial isometry with range $\ker V^*_T$ and $V_T$ an isometry, we obtain
\[M_{\Psi}M_{\Psi}^*=P_{\ker V^*_T}=I-V_TV_T^*.\]
Since $T$ was an arbitrary  pure $1/L$-contraction, we conclude that $L$ admits a characteristic function through $K,$ as desired.

For the converse, assume \(L\) admits a characteristic function through \(K\) and let
\(\mathcal{M}\subset\mathcal{H}_L\otimes\mathcal{E}\) be a closed shift-invariant
subspace, where \(\mathcal{E}\) is a Hilbert space. First, assume that $
\mathcal{M}\neq\mathcal{H}_L\otimes\mathcal{E}. $
Set $\mathcal{N}=\mathcal{M}^{\perp}$
and $S=M_z^L\otimes I_{\mathcal E}.$ Every
\(f\in\mathcal H_L\otimes\mathcal E\) has a unique expansion
\[
f=\sum_{\alpha}z^\alpha\otimes f_\alpha,
\]
where \(f_\alpha\in\mathcal E\). Moreover, for every
\(e\in\mathcal E\),
\[
\begin{aligned}
\langle (S^{*\alpha}f)(0),e\rangle
&=\langle S^{*\alpha}f,1\otimes e\rangle\\
&=\langle f,z^\alpha\otimes e\rangle\\
&=\frac1{l_\alpha}\langle f_\alpha,e\rangle,
\end{aligned}
\]
and therefore $
(S^{*\alpha}f)(0)=l_\alpha^{-1}f_\alpha.$
Now define
\[
\mathcal E_0
=
\overline{\operatorname{span}}
\{f(z):f\in\mathcal N,\ z\in\mathbb B_d\}.
\]
\begin{claim}
    $
\mathcal E_0
=
\overline{\operatorname{span}}\{f(0):f\in\mathcal N\}.$
\end{claim}
    \begin{proof}[Proof of Claim]
    Let
$
g=\sum_\alpha z^\alpha\otimes g_\alpha
\in\mathcal N.$
Since \(\mathcal N\) is co-invariant for \(S\),
$
S^{*\alpha}g\in\mathcal N.$
Hence,
\[
(S^{*\alpha}g)(0)
\in
\overline{\operatorname{span}}\{f(0):f\in\mathcal N\}, \quad \alpha\in\Z_+^d.
\]
Further,
\[
g_\alpha
=
l_\alpha(S^{*\alpha}g)(0)
\in
\overline{\operatorname{span}}\{f(0):f\in\mathcal N\}, \quad \alpha\in\Z_+^d,
\]
and so $g(z)
=
\sum_\alpha z^\alpha g_\alpha
\in
\overline{\operatorname{span}}\{f(0):f\in\mathcal N\}$ for all $z\in\mathbb{B}_d.$ Taking closed linear spans gives
\[
\mathcal E_0
\subseteq
\overline{\operatorname{span}}\{f(0):f\in\mathcal N\},
\]
proving the claim.
    \end{proof}
 Since every function in \(\mathcal N\) takes values in
\(\mathcal E_0\), we may replace \(\mathcal E\) by
\(\mathcal E_0\). Note that, since $\mathcal{N}\neq 0$, $ \dim \mathcal{E}_0\ge 1.$ 
Also, 
$
\mathcal N
\subseteq
\mathcal H_L\otimes\mathcal E_0,
$
and
\begin{equation}\label{Mdecomp}
\mathcal M
=
(\mathcal H_L\otimes\mathcal E_0\ominus\mathcal N)
\oplus
(\mathcal H_L\otimes\mathcal E_0^\perp).
\end{equation}
Now, define
$T
=
P_{\mathcal N}
S|_{\mathcal N}$. Since $S$ is a (pure) $1/L$-contraction and $\HN$ is co-invariant,  $T$ is a pure $1/L$-contraction as well.
Write
\[
\frac1L
=
1-
\sum_{\alpha\neq0}
c_\alpha
z^\alpha\overline w^\alpha,
\]
and define
\[
V:\mathcal N\to\mathcal E_0,
\qquad
Vf=f(0).
\]
$V$ is not the zero operator. Also, it is easily verified that  $\mathcal{H}_L\otimes\HE_0$ is reducing for $S$. Thus, we may regard, for the moment, $S$ as an operator defined on $\mathcal{H}_L\otimes\HE_0$, and write
\[
\begin{aligned}
\Delta_T^2
&=
P_{\mathcal N}
-
\sum_{\alpha\neq0}
c_\alpha
P_{\mathcal N}
S^\alpha
S^{*\alpha}
|_{\mathcal N}\\
&=
P_{\mathcal N}
\Big(
I-
\sum_{\alpha\neq0}
c_\alpha
S^\alpha
S^{*\alpha}
\Big)
|_{\mathcal N}\\
&=
P_{\mathcal N}
P_{1\otimes\mathcal E_0}
|_{\mathcal N}\\
&=
V^*V,
\end{aligned}
\]
where the last equality follows from $
P_{1\otimes\mathcal E_0}f
=
1\otimes f(0) $ (note that $||1||_{\mathcal{H}_L}=1$). Let
$
V=W|V| $
be the polar decomposition of $V$. Then
\[
|V|
=
(V^*V)^{1/2}
=
\Delta_T,
\]
and since
$
\overline{\operatorname{ran}V}
=
\mathcal E_0$ in view of our earlier claim,
the operator
\[
W:
\ol{\operatorname{ran}}\Delta_T
\to
\mathcal E_0
\]
is unitary. Next, consider the isometry
$
V_T:
\mathcal N
\to
\mathcal H_L
\otimes
\overline{\operatorname{ran}\Delta_T}$
with
\[
V_Th
=
\sum_\alpha
l_\alpha
z^\alpha
\otimes
\Delta_T
T^{*\alpha}h.
\]
Since $\HN$ is co-invariant for $S,$ we obtain 
\[
V_Th
=
\sum_\alpha
l_\alpha
z^\alpha
\otimes
\Delta_T
S^{*\alpha}h.
\]
Hence,
\[
(I\otimes W)V_Th
=
\sum_\alpha
l_\alpha
z^\alpha
\otimes
VS^{*\alpha}h.
\]
Using the coefficient identity proved above,
\[
VS^{*\alpha}h
=
l_\alpha^{-1}h_\alpha,
\]
and therefore
\[
(I\otimes W)V_Th
=
\sum_\alpha
z^\alpha
\otimes
h_\alpha
=
h.
\]
Thus,
$
(I\otimes W)V_T
=
P_{\mathcal N}.$
Since $
(I\otimes W)(\operatorname{ran}V_T)
=
\mathcal N,$ and  $I\otimes W:\mathcal{H}_{L}\otimes \ol{\operatorname{ran}}\Delta_T \to \mathcal{H}_{L}\otimes \HE_0 $ is unitary, we obtain
\begin{equation} \label{Wrange}
(I\otimes W)( \ker V_T^*)
=
\mathcal H_L\otimes\mathcal E_0
\ominus
\mathcal N.
\end{equation}
Now, since \(L\) admits a characteristic function through \(K\) and
\(\ker V_T^*\) is shift-invariant, there exist a Hilbert space
\(\mathcal F\) and a multiplier
\[
\Phi:
\mathcal H_K\otimes\mathcal F
\to
\mathcal H_L\otimes
\overline{\operatorname{ran}\Delta_T}
\]
such that
\[
P_{\ker V_T^*}
=
I-
V_TV_T^*
=
M_\Phi M_\Phi^*.
\]
In particular,  $M_{\Phi}$ is a partial isometry and $
\operatorname{ran}M_\Phi
=
\ker V_T^*.$ Define $
\Psi_0(z)=W\Phi(z),$
so that
\[
M_{\Psi_0}
=
(I\otimes W)M_\Phi.
\]
Since $I\otimes W$ is unitary, $\Psi_0$ yields a partially
isometric multiplier
\[
M_{\Psi_0}:
\mathcal H_K\otimes\mathcal F
\to
\mathcal H_L\otimes\mathcal E_0
\]
whose range, by \eqref{Wrange}, is
\[
\operatorname{ran}M_{\Psi_0}
=
(I\otimes W)(\operatorname{ran}M_\Phi)
=
(I\otimes W)(\ker V_T^*)
=
\mathcal H_L\otimes\mathcal E_0
\ominus
\mathcal N.
\]
\begin{lemma}\label{temp}
In the setting of the proof, $L/K\succeq 0.$
\end{lemma}
\begin{proof}[Proof of Lemma \ref{temp}]
Assume $\dim \HE=1,$ in which case $\dim \HE_0=1$ and thus $\HE_0=\HE.$ We have
$M_{\Psi_0}:
\mathcal H_K\otimes\mathcal F
\to
\mathcal H_L\otimes\mathcal E$ and 
\[
\operatorname{ran}M_{\Psi_0}
=
\mathcal H_L\otimes\mathcal E
\ominus
\mathcal N=\HM.
\]
Since $M_{\Psi_0}$ is partially isometric, we obtain, by Lemma \ref{coiso-mult}, 
\[\cfrac{L_{\mathcal{M}}}{K}\succeq 0 \]
whenever $\mathcal{M}\subsetneq \mathcal{H}_{L}\otimes\HE\cong \mathcal{H}_L$ is shift-invariant. In particular, by choosing \[\mathcal{M}=\mathcal{M}_{\lambda}=\{f\in\mathcal{H}_L \ |\ f(\lambda)=0\},\qquad \lambda\in\mathbb{B}_d,\] we obtain
\[\cfrac{L(z,w)-\cfrac{L(z, \lambda)L(\lambda, w)}{L(\lambda,\lambda)}}{K(z,w)}=\cfrac{L_{\mathcal{M}_{\lambda}}(z,w)}{K(z,w)}\succeq 0, \qquad z,w,\lambda\in\mathbb{B}_d. \]
Choosing $\lambda=\lambda_n$ with $\lambda_n\to\partial\mathbb{B}_d$ and letting $n\to\infty$ in the above inequality, we obtain, in view of item~\ref{Schurcompl} of Definition \ref{admissible-def}, $L/K\succeq 0,$ as desired.
\end{proof}

\begin{remark}
    Note that the proof of the reverse implication in Theorem \ref{mainthm} (i.e. the BLH property implies that $L$ admits a characteristic function through $K$) does not require items \ref{K=nonzero} and \ref{Schurcompl} of Definition \ref{admissible-def}.
\end{remark}

Returning to the main proof, let
\[
P_{\mathcal E_0^\perp}:\mathcal E\to\mathcal E_0^\perp
\]
denote the orthogonal projection. Multiplication by $P_{\mathcal E_0^\perp}$ defines a
partially isometric multiplier
\[
M_{P_{\mathcal E_0^\perp}}:
\mathcal H_L\otimes\mathcal E 
\to
\mathcal H_L\otimes\mathcal E_0^\perp
\]
whose range is $
\mathcal H_L\otimes\mathcal E_0^\perp.$ Finally, since $L/K\succeq 0$, Lemma \ref{coiso-mult} yields the existence of a Hilbert space $\mathcal{F}'$ and a partially isometric multiplier $\Psi_1:\mathcal{H}_{K}\otimes\HF'\to \mathcal{H}_{L}\otimes\HE$ such that $\operatorname{ran} M_{\Psi_1}= \mathcal{H}_{L}\otimes\HE$.
 Define the block operator $\Psi(z): \HF\oplus \HF'\to  \HE=\HE_0\oplus \HE_0^{\perp}$ by
\[
\Psi(z)
=
\begin{bmatrix}
\Psi_0(z) & 0\\
0 & P_{\mathcal E_0^\perp}\Psi_1(z)
\end{bmatrix},
\]
which yields a multiplier $
\Psi:
\mathcal H_K\otimes
(\mathcal F\oplus \mathcal{F}')
\longrightarrow
\mathcal H_L\otimes\mathcal E. $
Then, $M_\Psi$ is   partially isometric  and, in view of \eqref{Mdecomp},
\[
\operatorname{ran}M_\Psi
=
(\mathcal H_L\otimes\mathcal E_0\ominus\mathcal N)
\oplus
(\mathcal H_L\otimes\mathcal E_0^\perp)
=
\mathcal M,
\]
completing the proof. Note that the case
$\mathcal{M}=\mathcal{H}_L\otimes\mathcal{E} $ is addressed by the positivity condition $L/K\succeq 0,$ as we just observed.
\end{proof}
A straightforward application of Lemma~\ref{coiso-mult} obtains:
\begin{corollary} \label{L_M/K>>0}
Let $(K, L)$ be an admissible pair of kernels on $\mathbb{B}_d.$ Then, $L$ admits a characteristic function through $K$ if and only if, for every Hilbert space $\HE$
and every shift-invariant $\HM\subset\mathcal{H}_L\otimes \HE$,
\[\cfrac{L_{\mathcal{M}}}{K}\succeq 0.\]
\end{corollary}

 Note that Theorem~\ref{mainthm} says, in the language of \cite{BLHpairs}, that $L$ admits a characteristic function through $K$ if and only if $(K, L)$ has the BLH property, at least with respect to shift-invariant subspaces $\HM$. We now establish Theorem~\ref{BLHcorol} by invoking the characterization of the BLH property from \cite{BLHpairs}.
\begin{proof}[Proof of Theorem~\ref{BLHcorol}]
Assume $(K, L)$ is strongly admissible. Then, in the terminology of \cite{BLHpairs}, $(K, L)$ is strongly regular. In particular, any pair satisfying Definition \ref{strong-admissible-def} also satisfies \cite[Definition 2.8]{BLHpairs} with $\Omega=\mathbb{B}_d, \Omega'=t\mathbb{B}_d$ and $r_n(z)=r_n z.$ Note that $\mathbb{B}_d$ is complete Reinhardt, and so any holomorphic function on $\mathbb{B}_d\times\mathbb{B}_d$ is a uniform-on-compacta limit of polynomials. Moreover, by \cite[Corollary 3.5]{BLHpairs}, a closed $\mathcal{M}\subset\mathcal{H}_L\otimes\HE$ is shift-invariant if and only if it is $\textup{Mult}(\mathcal{H}_K)$-invariant if and only if it is $\textup{Mult}(\mathcal{H}_L)$-invariant. Now, Theorem~\ref{mainthm} says that $(K,L)$ admits a characteristic function if and only if it is a Beurling-Lax-Halmos pair. Since $1/K$ decomposes as the difference of two positive reproducing kernels (simply separate the negative from the positive coefficients in the power series expansion), \cite[Theorem 7.1]{BLHpairs} yields \eqref{Aglerdecomp}.
\end{proof}

There are important special cases where the decomposition \eqref{Aglerdecomp} can be considerably simplified. Recall that a kernel $K$ on $\mathbb{B}_d$ is a regular unitarily invariant kernel if it has the form
\[K(z, w)=\sum_{n=0}^{\infty}k_n\langle z, w\rangle^n \]
with $\lim_n k_{n+1}/k_n=1.$

\begin{proposition}\label{exampleProp}
Let $(K, L)$ be an admissible pair of kernels on $\mathbb{B}_d$ such that $L$ is regular unitarily invariant and 
\[K(z, w)=\cfrac{1}{(1-\langle z, w\rangle)^{\rho}}, \qquad z, w\in\mathbb{B}_d,\]
with $\rho>0.$ Then, $L$ admits a characteristic function through the kernel $K$ if and only if
$0 < \rho \leq 1$ and $L/K\succeq 0$. 
\end{proposition}
\begin{proof}
Suppose $L$ admits a characteristic function through $K$. Note that \[\HM_N: = \overline{\text{span}}\{z^\alpha: \alpha\in \Z^d_+, |\alpha| > N\},\]
is a shift-invariant subspace of $\mathcal{H}_{L}$, for any $N\ge 1.$ Thus, by Corollary \ref{L_M/K>>0},
\begin{align*}
\frac{L_{\HM_N}(z, w)}{K(z, w)} &= \frac{L(z, w)-\sum_{|\alpha| \leq N} \ell_{\alpha} z^\alpha \overline{w}^\alpha}{K(z, w)} \\
& = (1-\langle z,w\rangle)^{\rho} \sum_{n=N+1}^\infty \ell_n\langle z,w\rangle^n  \\
&\succeq 0.
\end{align*}
Thus, for any $N\ge 1,$ the power series
\[(1-x)^\rho \sum_{n=N+1}^\infty \ell_nx^n = \ell_{N+1} x^{N+1} + (\ell_{N+2} - \rho \ell_{N+1})x^{N+2} + \ldots\]
has nonnegative coefficients.
Since $L_{n+1}/L_n \rightarrow 1$, we conclude from \cite[Corollary 6.3]{HartzIsomorphism} that $0 < \rho \leq 1$. That $K$ is a factor of $L$ follows from Corollary \ref{L_M/K>>0}.

For the converse, observe that $0 < \rho \leq 1$ implies $K$ is complete Pick. Thus, $K$ is a complete Pick factor of $L.$ That $L$ admits a characteristic function through $K$  now follows from Corollary \ref{L_M/K>>0} and \cite[Lemma 2.1]{Beurlingfactor}.

\end{proof}

\printbibliography
\end{document}